\documentclass{amsart}

\title[Right ideals  and indecomposable permutations]{
Number of right ideals  and a $q$-analogue of indecomposable permutations}
\author{Roland Bacher}
\author{Christophe Reutenauer}

\address{Roland Bacher\\ Univ. Grenoble Alpes\\Institut Fourier (CNRS UMR 5582) \\ 100 rue des Maths\\F-38000 Grenoble\\France}\email{roland.bacher@fourier.ujf-grenoble.fr}

\address{Christophe Reutenauer\\D\'epartement de Math\'ematiques  UQAM\\Case Postale 8888 Succ. Centre-ville\\Montr\'eal (Qu\'ebec) H3C 3P8 Canada} \email{Reutenauer.Christophe@uqam.ca}

\thanks{}

\keywords{}
\subjclass[2000]{}

\usepackage{amscd}
\usepackage{verbatim}
\usepackage{tikz}
\usetikzlibrary{shapes}

\newtheorem{theorem}{Theorem}

\newtheorem{lemma}{Lemma}

\newfont{\Bbbb}{msbm10 at 12pt}

\usepackage{graphicx}
\usepackage{pstricks}

\date{\today}

\begin{document}

\maketitle

{\bf Abstract.}
We prove that the number of right ideals of codimension $n$ in 
the algebra of noncommutative Laurent polynomials in two variables 
over the finite field $\mathbb F_q$ is equal to 
$(q-1)^{n+1} q^{\frac{(n+1)(n-2)}{2}}\sum_\theta q^{inv(\theta)}$, where the 
sum is over all indecomposable permutations in $S_{n+1}$ and where $inv(\theta)$
stands for the number of inversions of $\theta$.
\vskip1cm
Mathematics Subject Classification Numbers: Primary: 05A15, Secondary: 05A19

\eject

\section{Introduction} \label{introduction}
Indecomposable permutations 
and subgroups of finite index of the free group $F_2$ are equinumerous.  More precisely, the latter number was computed by Hall \cite{H} and the former by Comtet \cite{Cm}, and it turns out that the number of subgroups of index $n$ is equal to the number of indecomposable permutations in $S_{n+1}$.To our best knowledge, this was first noticed by Dress and Franz \cite{DF1} who gave a bijective proof. Sillke \cite{Si}, Ossona de Mendez and Rosenstiehl \cite{OR}, and Cori \cite{Cr} discovered more bijections. We give a further bijection based on a natural 
correspondence between
subgroups of finite index and regular right congruences of free monoids
during the proof of our main result, Theorem \ref{bijection}. 

Note that subgroups of the free group $F_2$ are also naturally in bijection with rooted hypermaps, see \cite{Cr}. Indecomposable permutations appear also in the study of planar maps, see \cite{B}. Furthermore, they form a free generating set of the bialgebra of permutations, which is therefore a free associative algebra \cite{PR} and they index a basis of its primitive elements \cite{AS}. More elementary is the folklore result that the disjoint union of all permutations groups is a free monoid under shifted concatenation, freely generated by the set of indecomposable permutations (a proof may be found in \cite{P}. This yields a generating function (see \cite{C}) allowing to 
count indecomposable permutations easily.

Our main result is a linear version of the bijections mentioned above: We consider the polynomial $P_{n+1}(q)=\sum_{\theta\in \mathrm{Indec}_{n+1}} q^{inv(\theta)}$ enumerating indecomposable permutations by inversions. We show that this polynomial, corrected by the factor $(q-1)^{n+1} q^{\frac{(n+1)(n-2)}{2}}$, 
enumerates right ideals of finite index $n$ in the group algebra of the free group $\langle a,b\rangle$ on two generators over the finite field $\mathbb F_q$ (Theorem \ref{mainr}).

A key ingredient of our proof is a particular case of a result due to Haglund \cite{H}: The number of invertible matrices over $\mathbb F_q$ with support included in a fixed partition is given by a polynomial which counts, essentially by inversions, the number of 
permutation-matrices with the same support property. These polynomials are rook polynomials \cite{GR}:
They count the number of nonattacking positions of rooks on a chess board. 

A last ingredient of the proof is a study of prefix-free sets and prefix-closed sets with respect to the alphabetical order in the free monoid $\{a,b\}^*$ (or equivalently in binary trees). In particular, our Lemma \ref{numbers},
a somewhat technical formula linking different parameters, seems to be new.

This article is a sequel to \cite{BR}, where we have shown that the number of right ideals of index $n$ of the free associative algebra on two generators over $\mathbb F_q$ is a classical $q$-analogue of Catalan numbers, a result which was already implicit in Reineke's article \cite{R}. We use here several results of our 
previous paper: A
description by prefix-free and prefix-closed sets of right ideals in the free associative algebra, based on
the fact that this algebra is a fir (free ideal ring), in the sense of Cohn \cite{C}, together with a noncommutative version of Buchberger's algorithm for the 
construction of Gr\"obner bases.

{\it Acknowledgments}: A discussion with Alejandro Morales helped the second author to understand Haglund's theorem and rook polynomials, together with the variant given by him and his co-authors in \cite{BIMPSZ}.

\section{Main result} \label{main}

A permutation $\sigma \in S_n$ of the set 
$\{1,\dots,n\}$ is {\em decomposable} if $\sigma(\{1,\dots,i\})=\{1,\dots,i\}$ (and $\sigma(\{i+1,\ldots,n\})= \{i+1,\ldots,n\}$) for
some $i$ in $\{1,\dots,n-1\}$. A permutation $\sigma$ is {\em indecomposable}
otherwise, ie. if $\sigma(\{1,\dots,i\})\not=\{1,\dots,i\}$ 
for all $i$ in $\{1,2,\dots,n-1\}$.

An \emph{inversion} of a permutation $\sigma\in S_n$ is
a pair of distinct integers $i,j$ with $1\leq i<j\leq n$ 
such that  
$\sigma(i)>\sigma(j)$. We write $\mathop{inv}(\sigma)$
for the number of inversions of $\sigma$.

We denote by $K\langle a,b,
a^{-1},b^{-1}\rangle$ the ring of noncommutative Laurent polynomials in noncommuting variables 
$a,b$ over a field $K$. Equivalently, $K\langle a,b,
a^{-1},b^{-1}\rangle$ is the $K$-algebra of the free group generated by $a,b$.

\begin{theorem}\label{mainr}
Given a finite field $\mathbb F_q$, the number of right ideals having codimension $n$ of the $\mathbb F_q$-algebra $\mathbb F_q\langle a,b,a^{-1},b^{-1}\rangle$ 
of noncommutative Laurent polynomials in two free noncommuting generators $a,b$ is equal to 
\begin{eqnarray}\label{formulemainr}
(q-1)^{n+1} q^{\frac{(n+1)(n-2)}{2}}\sum_{\theta\in \mathrm{Indec}_{n+1}} q^{inv(\theta)},
\end{eqnarray}
with $\hbox{Indec}_{n+1}$ denoting
the subset of all indecomposable permutations in $S_{n+1}$. 
Equivalently,
formula (\ref{formulemainr}) is also given by 
\begin{eqnarray}\label{formulemainreq}
\left(\frac{q-1}{q}\right)^{n+1}\sum_{\theta\in \mathrm{Indec}_{n+1}} q^{p(\theta)},
\end{eqnarray}
where $p(\theta)$ is equal to the cardinality of the set $\{(i,j), 1\leq i,j\leq n+1, j<\theta(i)$ or $i>\theta^{-1}(j)\}$.
\end{theorem}

The first polynomials $P_{n+1}(q)=\sum_{\theta\in \mathrm{Indec}_{n+1}} q^{inv(\theta)}$ corresponding to $n=0,1,2,3$ are $1$, $q$, $q^3+2q^2$ and $q^6+3q^5+5q^4+4q^3$.
 
 The generating series of all these polynomials is easy to compute as follows: 
The shifted concatenation $\alpha\cdot \beta
\in S_{m+n}$ of two permutations $\alpha\in S_m$ and $\beta\in S_n$
is the permutation of $\{1,\dots,n+m\}$ defined by
$\alpha\cdot \beta(i)=\alpha(i)$ if $i\in \{1,\dots,m\}$ and 
by $\alpha\cdot \beta(i)=m+\beta(i-m)$ if $i\in \{m+1,\dots,m+n\}$.
The disjoint union $\mathcal M=\cup_{n\in\mathbb N}S_n$ (with $S_0$ 
acting on the empty set), endowed with shifted concatenation,
is the free (noncommutative) monoid generated by the 
set $\cup_{n=1}^\infty \mathrm{Indec}_n$ of all indecomposable permutations, see \cite{Cm}.
Since $\mathop{inv}(\alpha\cdot \beta)=\mathop{inv}(\alpha)+\mathop{inv}(\beta)$, the map $\mathop{inv}:\mathcal M\longmapsto \mathbb N$ defines a 
morphism from the monoid $\mathcal M$ onto the additive monoid $\mathbb N$.
Freeness of $\mathcal M$ over $\cup_{n=1}^\infty \mathrm{Indec}_n$
implies the formula
\begin{eqnarray*}
\sum_{n\in\mathbb N}t^n\sum_{\sigma\in S_n}q^{\mathop{inv}(\sigma)}&=&
\sum_{k=0}^\infty\left(\sum_{n=1}^\infty t^n\sum_{\sigma\in \mathrm{Indec}_n}
q^{\mathop{inv}(\sigma)}\right)^k\\
&=&\frac{1}{1-\sum_{n=1}^\infty t^n\sum_{\sigma\in \mathrm{Indec}_n}
q^{\mathop{inv}(\sigma)}}
\end{eqnarray*}
for the generating series of the number of permutations with a given 
number of inversions. An easy induction yields the well-known identity
$$\sum_{\sigma\in S_n}q^{\mathop{inv}(\sigma)}
=(1)(1+q)\cdots (1+q+\cdots +q^{n-1})=(q-1)^{-n}\prod_{k=1}^n(q^k-1)$$
where the right-hand side involves
the classical $q-$analogue of the factorial function. We have therefore
$$
\sum_{n\geq 0} (1)(1+q)\cdots(1+q+\cdots+q^{n-1})t^n=(1-\sum_{k\geq 1}P_k(q)t^k)^{-1}.
$$

\section{Inversions and hooks}\label{hook}

We represent a permutation $\sigma\in S_n$ by its permutation matrix, with a 1 in row  $i$ and column $\sigma(i)$, for each $i, 1\leq i \leq n$, and 
with
0's elsewhere. (With this convention, well-suited to the group-structure
of $S_n$, a permutation matrix $M_\sigma$ acts by right-multiplication 
$v\longmapsto v\cdot M_\sigma$ through the coordinate-permutation
 $v_i\longmapsto v_{\sigma(i)}$ on a row-vector $v=(v_1,\dots,v_n)$.)
A {\em hook} of a coefficient 
$1$ in such a matrix is the set 
of 0's which are located either on the same row at its left or
on the same column and below it. In other words, the hook associated to
$(i,\sigma(i))$
is the set of entries with coordinates $(i,j), j<\sigma(i)$, or $(k,\sigma(i)), k>i$.
Since the permutation matrix of $\sigma^{-1}$ is obtained by transposing 
the permutation matrix of $\sigma$, the cardinality of the union of all hooks in the permutation matrix of $\sigma$ is equal to the number 
$p(\sigma)$ introduced in Theorem \ref{mainr}. Hooks are in general not disjoint. They can be enumerated as follows:
An inversion of $\sigma$ gives rise to a pair of 1's (in the permutation matrix of $\sigma$)
which are incomparable with respect to the order on entries induced by the natural order of $\mathbb N\times \mathbb N$ (ie. one of the coefficients $1$ has
a higher row-index, the other a higher column-index). 
A {\em version} of $\sigma$ is a pair $i<j$ of distinct integers such that 
$\sigma(i)<\sigma(j)$. Equivalently, a version corresponds to a pair of 1's (in the permutation matrix of $\sigma$) whose entries are comparable for the natural order on
$\mathbb N\times \mathbb N$. Denote by $inv(\sigma)$ and by $v(\sigma)$ the number of inversions and of versions of $\sigma$. 
We have
\begin{eqnarray}\label{eqpsigma}
&p(\sigma)=2inv(\sigma)+v(\sigma)=inv(\sigma)+\binom{n}{2}.&
\end{eqnarray}
The first equality follows from Figure \ref{hooks}, which shows that each inversion increases $p(\sigma)$ by $2$ (left part of Figure \ref{hooks}) whereas  a version adds only $1$ corresponding to a unique $0$ contained 
in both hooks (right part of Figure \ref{hooks}).
The observation that the sum $inv(\sigma)+v(\sigma)$ is equal to the number $\binom{n}{2}$ of all 2-subsets in $\{1,\dots,n\}$ shows the second equality.

\begin{figure}
$$
\begin{array}{cccccccccccccccccccccccccccccccccccccc}
* & \cdots & 1 &&&&&&&&&&&&&& 1\\
\vdots & &\vdots &&&&&&&&&&&&&& \vdots \\
1 &\cdots & *  &&&&&&&&&&&&&& * & \cdots & 1
\end{array}
$$
\caption{}\label{hooks}
\end{figure}

Rewriting the equality $p(\theta)=inv(\theta)+\binom{n+1}{2}$ as $p(\theta)-n-1=inv(\theta)+\frac{(n+1)(n-2)}{2}$ we deduce equality of Formulae
(\ref{formulemainr}) and (\ref{formulemainreq}) in Theorem \ref{main} from (\ref{eqpsigma}) and from the trivial identity $\frac{(n+1)(n-2)}{2}+n+1=\frac{n(n+1)}{2}$.

\section{A theorem of Haglund}\label{thHag}
Let $\lambda=(\lambda_1,\ldots, \lambda_l)$ be a partition, with $0\leq\lambda_1\leq\ldots\leq\lambda_n\leq n$. We associate to it the subset $E_\lambda$ of $[n]\times[n]$ defined by $E_\lambda=\cup_{1\leq i\leq n}\{i\}\times[\lambda_i]$, where $[k]=\{1,\dots,k\}$.
The following result is a particular case of Theorem 1 in \cite{Hg}.
\begin{theorem}\label{hag}
The number of $n\times n$ invertible matrices over $\mathbb F_q$ whose nonzero entries lie in $E_\lambda$ is equal to $(q-1)^{n}\sum_\sigma q^{p(\sigma)}$, where the sum is over all permutations in $S_n$ whose matrices have nonzero entries only in $E_\lambda$.
\end{theorem}
We have stated the variant of \cite{BIMPSZ}, the actual formulation in \cite{Hg} is slightly different.

For further use, we denote by $H_\lambda(q)$ the polynomial appearing in
Theorem \ref{hag}. The polynomial $H_\lambda(q)$ is of course zero if
$\lambda_i<i$ for some $i\in \{1,\dots,n\}$.

The polynomial $H_\lambda(q)$ of Haglund's theorem is closely related to rook polynomials as defined and studied in \cite{GR}. Such polynomials are symmetric (self-reciprocal). Observe however that this is not the case for the polynomials appearing in Theorem \ref{mainr}.

We illustrate Haglund's theorem by an example: 
The four possible invertible permutation matrices (corresponding to the permutations
$123, 132,213, 231$)  with non-zero entries in the partition 
$\lambda=(2,3,3)$ (given at the left of the following figure) 
are depicted in the figure below.
$$
\left( 
\begin{array}{ccc}
\times&\times&0\\
\times&\times&\times\\
\times&\times&\times 
\end{array}
\right)
\left( 
\begin{array}{ccc}
1&0&0\\
\bf 0&1&0\\
\bf 0&\bf 0&1 
\end{array}
\right)
\left( 
\begin{array}{ccc}
1&0&0\\
\bf 0&\bf 0&1\\
\bf 0&1&\bf 0 
\end{array}
\right)
\left( 
\begin{array}{ccc}
\bf 0&1&0\\
1&\bf 0&0\\
\bf 0&\bf 0&1  
\end{array}
\right)
\left( 
\begin{array}{ccc}
\bf 0&1&0\\
\bf 0&\bf 0&1 \\
1&\bf 0&\bf 0 
\end{array}
\right)
$$ 
The function $p(\sigma)$ has respectively the values 3,4,4 and 5, as shown by the boldfaced 0's which represent the union of the hooks.
The corresponding polynomial $H_\lambda(q)$ is given by $(q-1)^3q^3(1+q)^2$.

\subsection{Proof of Theorem \ref{hag}}

We present a short proof of Theorem \ref{hag} for the 
sake of self-containedness:

Given a partition $\lambda$ with $n$ parts $\lambda_1\leq \lambda_2\leq \dots\leq \lambda_n=n$
satisfying $\lambda_i\geq i$, we claim that we have
\begin{eqnarray}\label{otherformulaHlambda}
H_\lambda(q)=q^{{n\choose 2}}\prod_{i=1}^n\left(q^{\lambda_i+1-i}-1\right).
\end{eqnarray}

For example, the partition $\lambda=(2,3,3)$ considered above
yields 
$$q^{{3\choose 2}}\left(q^{2+1-1}-1\right)\left(q^{3+1-2}-1\right)\left(q^{3+1-3}-1\right)=q^3(q-1)^3(q+1)^2,$$
as expected.

Formula (\ref{otherformulaHlambda}) is clearly true if $\lambda$ is reduced to a unique part $\lambda_1=1$
where an invertible matrix associated to $\lambda$ is simply a non-zero 
element of $\mathbb F_q$.
Consider now an invertible matrix $A$ compatible with $\lambda$
(and having its coefficients in $\mathbb F_q$). There are $q^{\lambda_1}-1$
possibilities for its first row. Let $j_1$ be the column-index
of the last non-zero coefficient
of the first row. Using the non-zero coefficient of row $1$ and column $j_1$ 
for Gaussian elimination, we can eliminate all 
non-zero coefficients in the remaining rows of the $j_1$-th column 
by subtracting a suitable 
multiple of the first row. All $q^{n-1}$ possibilities
for such an elimination can arise in a suitable invertible matrix $A$. 
Erasing in the resulting matrix the 
first row and the $j_1$-th column, we get an invertible matrix 
$A'$ associated to the partition $\lambda'$ with $n-1$ parts $\lambda_2-1,\dots,
\lambda_n-1$. We have thus by induction 
\begin{eqnarray*}
H_\lambda(q)&=&(q^{\lambda_1}-1)q^{n-1}H_{\lambda'}(q)\\
&=&(q^{\lambda_1}-1)q^{n-1}q^{n-1\choose 2}\prod_{i=2}^n\left(q^{\lambda_i+1-i}-1\right)
\end{eqnarray*}
which simplifies to (\ref{otherformulaHlambda}).

We associate now to the above matrix $A$ the permutation matrix 
with a $1$ in the $j_1$-th column of the first row. The remaining
non-zero coefficients are defined recursively as the permutation matrix
of $A'$ after removal of the first row and the $j_1$-th column.
The hook of the coefficient $1$ in the first column yields a contribution 
of $j_1-1+n-1$ to the number $p(\sigma)$ of the permutation $\sigma$
corresponding to the permutation matrix above. The identity
$p(\sigma)=j_1-1+n-1+p(\sigma')$, another induction on $n$
and a sum over all possibilities for $j_1$
imply now equality between Formula (\ref{otherformulaHlambda})
and the expression $(q-1)^n\sum_{\sigma}q^{p(\sigma)}$
given by Haglund's Theorem.\hfill$\Box$

{\bf Remark.} Formula (\ref{otherformulaHlambda}) is in fact much more
suited for computing  $H_\lambda(q)$ than the expression given in Theorem \ref{hag}.

\section{Prefix-free and prefix-closed sets}\label{sectprefixfreeclosed}

We denote by $A^*$ the free monoid over a finite set $A$.
Elements of $A^*$ are \emph{words}
with letters in the \emph{alphabet} $A$. The 
product of two words $u=u_1\dots u_n$ and $v=v_1\dots v_m$ in $A^*$
is given by the concatenation $uv=u_1\dots u_nv_1\dots v_m$.
The identity element of $A^*$ is the empty word, denoted 
by $1$ in the sequel. We use $a^*=\{a^n,n\geq 0\}$ 
for the set of all powers of 
a letter $a$ in the alphabet $A$.

A word $u$ is a {\em prefix} of a word $w$ if $w=uv$ for some word $v$. A subset $C$ of $A^*$ is {\em prefix-free} if no element of $C$ is 
a proper prefix of another element of $C$. 
A prefix-free set $C$ is {\em maximal} if it is not
contained in a strictly larger prefix-free set. A prefix-free 
set $C$ is maximal if and only if the right ideal $CA^*$ intersects every 
(non-empty) right ideal $I$ of the monoid $A^*$ \footnote{A right ideal 
of a monoid $\mathcal M$ is of course defined in the obvious way
as a subset $I$ of $\mathcal M$ such that $I\mathcal M=I$.}.
Indeed a prefix-free set $C$ giving rise 
to a right ideal $CA^*$ not intersecting a right ideal $I$ of $A^*$ can be 
augmented by adjoining an element of $I$. Conversely, a prefix-free set 
$C$ strictly contained in a prefix-free set $C\cup\{g\}$ 
defines a right ideal $CA^*$ which is disjoint from
the right ideal $gA^*$. Another characterization of maximal prefix-free
sets is given by the fact that a prefix-free set $C$ is maximal if and only
if every element of $A^*\setminus C$ has either an element 
of $C$ as a proper 
prefix or is a proper prefix of an element of $C$.

A subset $P$ of $A^*$ is {\em prefix-closed} if $P$ contains all prefixes of its elements. Equivalently, $P\subset A^*$ is prefix-closed if $u\in P$ 
whenever there exists 
$v\in A^*$ such that $uv\in P$. In particular, 
every non-empty prefix-closed set contains the 
empty word. 

There is a canonical bijection between finite maximal prefix-free sets and finite prefix-closed sets:  The prefix-closed set corresponding to a finite
maximal prefix-free set $C$ is the set $P=A^*\setminus CA^*$ of 
proper prefixes of all words in $C$. The inverse bijection associates to a finite prefix-closed set $P$ the finite maximal prefix-free set $C=PA\setminus P$ if $P$ is nonempty, and $C=\{1\}$ if $P$ is empty.
This bijection has the following graphical interpretation: Prefix-closed sets
have a natural rooted tree-structure: the root is the empty word, ancestors
of a vertex-word are its prefixes. The set $C\cup P$ (with $C$ and $P$ as 
above) is of course prefix-closed and the associated tree has leaves indexed by 
elements of $C$ and interior vertices indexed by elements of $P$. Every interior vertex has
the same number of children indexed by the alphabet $A$. The perhaps empty subtree 
of all interior vertices indexed by $P$ determines (and is uniquely determined by) 
the set of all leaves corresponding to $C$.

Consider now a finite maximal prefix-free set $C\subset \{a,b\}^*$ with associated finite prefix-closed set $P$. For $x\in \{a,b\}$, denote by $C_x\cap \{a,b\}^*x$
the set of words in $C$ terminating with $x$, and denote by $P_x=P\cap(\{1\}\cup\{a,b\}^*x)$ 
the subset of $P$ given by the union of $1$ (provided $P$ is 
non-empty) with
the subset of all non-empty words in $P$ ending with $x$. If $C\not=\{1\}$, we have a bijection $\mu_a$ between $C_a$ and $P_b$ given by $C_a\ni 
w\longmapsto p\in P_b$ where $p$ is the unique element of $P_b$ such that $w\in pa^*$.
The inverse bijection associates to $p\in P_b$ the unique word
$w=pa^*\cap C_a$ of $C_a$. 

Similarly, we have a bijection $\mu_b$ between $C_b$ and $P_a$ if $C\not=\{1\}$. 

Observe that $C$ is given by the disjoint union $C_a\cup C_b$,
except in the trivial case $C=\{1\}$ where we have $C_a=C_b=P=\emptyset$.
Assuming, as always in the sequel, that $C$ is nontrivial, we can
define $\mu: C \rightarrow P$ by using $\mu_a$ on $C_a$ and $\mu_b$ on $C_b$. 
Otherwise stated, $\mu(c)$ is the prefix of $c$ obtained by removing from $c$ its suffix of maximal length equal to a power of its last letter. Equivalently, $\mu(x)$ (for $x\not=1$)
is defined as the shortest proper prefix of $x$ such that $x\in \mu(x)a^*\cup\mu(x)b^*$.

Notice that $\mu(a^\alpha)=\mu(b^\beta)=1$
where $a^\alpha$ and $b^\beta$ are the two unique elements of 
$C$ involving only one letter. Notice also that $\mu$
induces a bijection between $C\setminus\{a^\alpha,b^\beta\}$
and $P\setminus\{1\}$ and that $\mu$ restricted to $C\setminus \{b^\beta\}$ is a bijection onto $P$.

Consider the prefix-free set $C=\{a^2,ab,ba^2,bab,b^2a,b^3\}$, represented by the leaves of a complete binary tree, see Figure \ref{pfset}. Its set of prefixes $P=\{1,a,b,ba,b^2\}$ is the set of internal nodes. One has $P_a=\{1,a,ba\}$, $P_b=\{1,b,b^2\}$, $C_a=\{a^2,ba^2,b^2a\}$ and $C_b=\{ab,bab,b^3\}$. The bijection $\mu_a$ sends $a^2,ba^2,b^2a$ respectively onto $1,b,b^2$ and $\mu_b$ sends $ab,bab,b^3$ respectively onto $a,ba,1$.

\begin{figure}\label{pfset}
\begin{tikzpicture}
[
auto,
level 1/.style={level distance=10mm, sibling distance=20mm},
level 2/.style={level distance=10mm, sibling distance=10mm},
level 3/.style={level distance=10mm, sibling distance=10mm},
level 4/.style={level distance=10mm, sibling distance=6mm},
v/.style={draw=black, text=black, circle, fill, inner sep=1.5},
f/.style={draw=black, text=black, rectangle, inner sep=2},
]

\node[v] (r) {} 
	child[v] { node[v] {} 
				child[sibling distance=6mm] {node[f] {}}
				child[sibling distance=6mm] {node[f] {}}
			}
	child[v] { node[v] {} 
				child[sibling distance=15mm] {node[v] {}
					child[sibling distance=6mm] {node[f] {}
					}
					child[sibling distance=6mm] {node[f] {}
					}
				}
				child[sibling distance=15mm] {node[v] {}
					child[sibling distance=6mm] {node[f] {}}
					child[sibling distance=6mm] {node[f] {}}
				}
			}
;
\end{tikzpicture}
\caption{A prefix-free set}\label{pfset}
\end{figure}
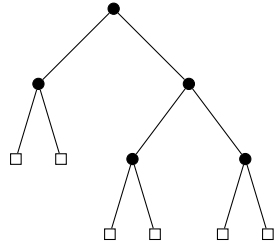

For all these results, see \cite{BR} or \cite{BPR}, where prefix-free sets are called prefix sets\footnote{A prefix-free set, not equal to $\{1\}$, is a {\em code}.}.

\section{Order properties of prefix-free and prefix-closed sets}

Let $A$ be a totally ordered alphabet. The {\em alphabetical} (or {\em lexicographical}) order on the free monoid $A^*$ is the order of the dictionary. Formally, one has $u<v$ if and only if $u$ is a proper prefix of $v$ or if $u=xay,v=xbz$ for some words $x,y,z$ and two distinct letters $a,b$ ordered by $a<b$ of the alphabet $A$.

For any words $u,v,x,y$ we have the following properties:
\begin{itemize}
\item $u<v$ if and only if $xu<xv$;
\item If $u$ is not a prefix of $v$ then $u<v$ implies $ux<vy$.
\end{itemize}

\begin{lemma} \label{ineq}
If $p,q\in\{a,b\}^*$ are two elements with $p\leq q$ lexicographically
then every element of $pa^*=\cup_{n=0}^\infty pa^n$ 
is lexicographically strictly smaller than every element of 
$qbb^*=\cup_{n=1}^\infty qb^n$.
\end{lemma}

\begin{proof}
If $p$ is not a prefix of $q$, we are done by the previous observation. If $q=pw$ then $pa^i<pwb^j$ if and only if $a^i<wb^j$. This holds for $j\geq 1$
since $a^i$ is strictly smaller than
any word involving $b$.
\end{proof}

A finite maximal prefix-free set $C$ of $\{a,b\}^*$ defines a complete
finite binary tree $T$ with leaves indexed by $C$ and interior vertices indexed by the associated prefix-closed set $P$. 
An element $v\in C_a$ defines a \emph{left branch} $\mu_a(v)a^*
\cap \left(C\cup P\right)$ of $T$. The natural integer 
$\mathrm{length}(v)-\mathrm{length}(\mu_a(v))$ is the \emph{length}
of the left-branch defined by $v\in C_a$ 
(see also the last paragraph of Section \ref{twisting}).

We shall need the following arithmetic characterization of complete binary trees and associated prefix sets.

\begin{lemma}\label{charactree}
A complete finite binary tree $T$ with leaves indexing a finite maximal prefix-free set $C$ (and interior vertices defining the associated prefix-closed
set $P$) is uniquely determined by the ranks $i_1,\ldots,i_k$ of all
$k$ elements in $C_a$ among the alphabetically ordered elements of $C$
and by the corresponding lengths $l_1,\ldots,l_k$ of the associated
left branches.
\end{lemma}

The integers $i_1,\dots,i_k$ and $l_1,\dots,l_k$ of Lemma 2
can also be described as follows:
if $C=\{c_1,\ldots,c_{n+1}\}$
with $c_1<c_2<\dots<c_n<c_{n+1}$, then $C_a=\{c_{i_1},\dots,c_{i_k}\}$
and $l_j=\mathrm{length}(c_{i_j})-\mathrm{length}(\mu_a(c_{i_j}))$.
In the running example of Figure \ref{pfset} we have $k=3$, $l_1=2$, $l_2=2$, $l_3=1$, $i_1=1$, $i_2=3$, $i_3=5$. We leave the proof of Lemma \ref{charactree}
to the reader.

We establish for later use the following identity involving some 
numbers associated to prefix sets.

\begin{lemma} \label{numbers} We consider a finite maximal prefix-free set
$C\subset \{a,b\}^*$ with associated prefix-closed set $P$ having 
$n=\vert P\vert=\vert C\vert-1\geq 0$ elements. 
We write 
$$\tilde C_a=C\setminus C_b=\left\lbrace\begin{array}{ll}
C_a\qquad&\hbox{if }C\not=1,\\
\{1\}\qquad&\hbox{if } C=\{1\}
\end{array}\right.$$
for the set of all words in $C$ which do not end with $b$.
We denote by $i_1,\dots,i_k$ the  
ranks in $C=\{c_1,\dots,c_{n+1}\}$ (listed in alphabetical order) of all $k=
\vert\tilde C_a\vert$ elements in 
$\tilde C_a=\{c_{i_1},\dots,c_{i_k}\}\subset C$.
We introduce the set $\mathcal M=\{(p,c)\in (P_b\setminus \{1\}) \times C_b,\ p<c\}$
containing $M=\vert \mathcal M\vert$ elements. Then
$$M+\sum_{h=1}^ki_h=(n+1)(k-1)+k-\frac{k(k-1)}{2}\ .$$
\end{lemma}

In our running example of Figure \ref{pfset}, we have 
$\mathcal M=\{(b,bab),(b,b^3),(b^2,b^3)\}$ and
$M=3,\ n=5,\ k=3,\ i_1=1,\ i_2=3,
i_3=5$. Hence we get $3+1+3+5=12$ for the left side and 
$6\cdot 2+3-\frac{3\cdot 2}{2}=12$ for the right side.

\begin{proof} For $C=\{1\}$ the set $P$ is empty, 
we have $M=0,\ n=0,\ k=1,\ i_1=1$ 
and both sides of the identity are equal to $1$.

For $C\not=\{1\}$, we apply induction using the partition 
$C=aC'\cup bC''$. We have two subcases given by $n''=0$ 
(corresponding to $C''=\{1\}$) and by $n''>0$ (corresponding
to $C''\not=\{1\}$) using obvious notations.

In the case $n''=0$ we have the partition
$$
\mathcal M=a\mathcal M'\cup (a(P_b'\setminus \{1\})\times\{b\})
$$
(where $a\mathcal M'=\{(ap',ac'),(p',c')\in\mathcal M'\}$)
for the set $\mathcal M$. This shows
$M=\vert \mathcal M\vert=M'+\vert (P_b'\setminus \{1\})\vert$.
Existence of the bijection $\mu_a$ between $C'_a$ and $P'_b$ 
(which exists even if $C'=\{1\}$ since then both sets are empty) 
and the equality
$\vert (P'_b\setminus \{1\})\vert=\vert\tilde C'_a\vert -1=k'-1$
(holding also for $C'=\{1\}$ with our definition of 
$\tilde C'_a$) give $M=M'+k'-1$.
Since $k=k',\ i_h=i_h'$ for $h=\{1,\dots,k\}$ and since $n=n'+1$ we have
by induction
\begin{eqnarray*}
M+\sum_{h=1}^k i_h&=&k-1+M'+\sum_{h=1}^{k'} i'_h\\
&=&k-1+(n'+1)(k'-1)+k'-\frac{k'(k'-1)}{2}\\
&=&(n'+2)(k'-1)+k'-\frac{k'(k'-1)}{2}\\
&=&(n+1)(k-1)+k-\frac{k(k-1)}{2}\\
\end{eqnarray*}
which ends the proof in the case $C''=\{1\}$.

In the remaining case $n''>0$ corresponding to $C''\not=\{1\}$ we have
the partition
\begin{eqnarray*}
\mathcal M&=&a\mathcal M'\cup (a(P'_b\setminus \{1\})\times bC''_b)\cup b\mathcal M''\cup
(\{b\}\times bC''_b)\\
&=&a\mathcal M'\cup b\mathcal M''\cup((a(P_b'\setminus \{1\})\cup\{b\})\times bC''_b)
\end{eqnarray*} 
for $\mathcal M=\{(p,c)\in (P_b\setminus \{1\})\times C_b,\ p<c\}$ where
$a\mathcal M'=\{(ap',ac'),(p',c')\in\mathcal M'\}$ and
$b\mathcal M''=\{(bp'',bc''),(p'',c'')\in\mathcal M''\}$. Existence of the bijection $\mu_a$
between $C'_a$ and $P'_b$ shows $\vert (a(P_b'\setminus \{1\})
\cup\{b\}\vert=k'$.
Existence of the partition $C''=C''_a\cup C''_b$ (since $n''>0$)
together with the identities $\vert C''\vert =n''+1$ and 
$\vert C''_a\vert=k''$
implies $\vert C''_b\vert=n''+1-k''$. 
We have thus the identity
\begin{eqnarray*}
M&=&M'+M''+k'(n''+1-k'')
\end{eqnarray*}
holding also in the case $n'=0$ corresponding
to $C'=\{1\}$.

We have $k=k'+k'',\ i_h=i'_h$ for $h\in\{1,\dots,k'\}$
and $i_h=n'+1+i''_{h-k'}$ for $h\in\{k'+1,\dots, k'+k''\}$.
This implies
\begin{eqnarray*}
\sum_{h=1}^ki_h&=&\sum_{h=1}^{k'}i'_h+\sum_{h=k'+1}^{k'+k''}(n'+1+i''_{h-k'})\\
&=&\sum_{h=1}^{k'}i'_h+(n'+1)k''+\sum_{h=1}^{k''}i''_h
\end{eqnarray*}
and we have thus by induction
\begin{eqnarray*}
M+\sum_{h=1}^k i_h&=&M'+M''+k'(n''+1-k'')+\sum_{h=1}^{k'}i'_h+(n'+1)k''+\sum_{h=1}{k''}i''_h\\
&=&(n'+1)(k'-1)+k'-\frac{k'(k'-1)}{2}+\\
&&\quad +(n''+1)(k''-1)+k''-\frac{k''(k''-1)}{2}+\\
&&\quad+k'(n''+1-k'')+(n'+1)k''\\
\end{eqnarray*}
which simplifies to
\begin{eqnarray*}&&
(n'+n''+1+1)(k'+k''-1)+(k'+k'')-\frac{(k'+k'')(k'+k''-1)}{2}\\
&=&(n+1)(k-1)+k-\frac{k(k-1)}{2}
\end{eqnarray*}
as required.
\end{proof}

{\bf Remark.} Lemma \ref{numbers}, rewritten as
\begin{eqnarray}\label{lem3bij}
(n+1)k=\sum_{h=1}^ki_h+M+(n+1-k)+\frac{k(k-1)}{2}\ ,
\end{eqnarray}
has also the following bijective proof:
The left side of (\ref{lem3bij}) is the cardinality of the set
$C\times \tilde C_a$. The right side 
is the cardinality of the disjoint union
$E=E_1\cup E_2\cup E_3\cup E_4$, where
\begin{eqnarray*}
E_1&=&\{(c_1,c_2)\in C\times \tilde C_a,c_1\leq c_2\},\\
E_2&=&\mathcal M=\{(p,c)\in (P_b\setminus 1)\times C_b,\ p<c\},\\
E_3&=&(C\setminus \tilde C_a)=C_b,\\
E_4&=&\{(c_1,c_2)\in \tilde C_a\times \tilde C_a,c_1<c_2\}.
\end{eqnarray*}
The set $F=C\times \tilde C_a$ can be partionned into
$F=F_{\leq}\cup F_>$ with 
\begin{eqnarray*}
F_{\leq}&=&E_1=\{(c_1,c_2)\in C\times \tilde C_a,c_1\leq c_2\},\\
F_{>}&=&\{(c_1,c_2)\in C\times \tilde C_a,c_1> c_2\}.
\end{eqnarray*}
We leave it to the reader to check that
$\varphi:F_>\longrightarrow E_2\cup E_3\cup E_4$ given by
$$\varphi(c,\gamma)=\left\lbrace\begin{array}{ll}
(\mu_a(\gamma),c)\in E_1\qquad&c\in C_b,\ \gamma\not\in aa^*,\\
c\in E_3&c\in C_b,\ \gamma\in aa^*,\\
(\gamma,c)\in E_4&c\in C_a.
\end{array}\right.
$$
defines a bijective map.

\section{Twisted alphabetical order}

\subsection{Twisting a total order}\label{twisting}
Suppose that we have a set $E$ with a partition $E=\cup_{i \in I} E_i$, where $I$ and each $E_i$ are totally ordered. This gives a natural total order on $E$ 
by setting $x<y$ if either $x$ and $y$ with $x<y$ belong to a common subset $E_i$
or if $x\in E_i$ and $y\in E_j\not=E_i$ with $i<j$.

Call a subset $I$ of an ordered set $E$ an {\em interval} if $a\in I, b\in E, c\in I$ and $a<b<c$ implies $b\in I$. A set $I$ indexing disjoint non-empty intervals $E_i$ partitioning a totally ordered set $E=\cup_{i\in I}E_i$
is naturally ordered as follows: Given two distinct elements $i,j$ of $I$, we set $i<j$ if $x<y$ for some $x\in I_i, y\in I_j$.
Since the sets $E_i$ are intervals, this is a well-defined
total order relation on $I$, independent of the 
chosen representatives $x$ and $y$. 
We use this partition and the previous construction to define the {\em twisted total order $\prec$ (with respect to the partition $\cup_{i\in I} E_i$)}: The restriction of $\prec$ to each $E_i$ is the opposite order of $<$ on $E_i$ and the set $I$ is ordered by $<$.

\medskip
{\bf Remark.}  (i) It is also possible to twist the order on $E=\cup_{i\in I}$
according to the set of indices: $x\tilde\prec y$ if either $x<y$ for $x,y\in E_i$ 
or $x\in E_i,y\in E_j$ with $i>j$. Twisting an order relation on the set of 
indices of a suitable partition amounts however to the ordinary order twist
of the opposite order relation with respect to the same partition.

(ii) Twisted orders can be generalized to arbitrary (not necessarily) 
totally ordered) posets using \emph{admissible} partitions indexed by posets
where a partition $E=\cup E_i$ of a poset $E$ is admissible if 
all elements in any common part $E_i$ have the same sets of upper and lower bounds in $E\setminus E_i$.

\medskip
Consider now $\{a,b\}^*$ with the alphabetical order. We partition
$\{a,b\}^*$ into equivalence classes given by 
$u\sim v$ if $ua^*\cap va^*\not=\emptyset$. Elements in a common equivalence
class
differ thus at most by a final string of $a$'s. Each equivalence class
can be written as
$wa^*$ for a unique word $w$ in $\{a,b\}^*\setminus \{a,b\}^*a=\{1\}\cup \{a,b\}^*b$. More precisely, the equivalence class of an element $w$ is the set $wa^*$ 
if $w\in \{a,b\}^*b\cup\{1\}$ does not end with $a$. 
It is given by $\mu_a(w)a^*$
if $w\in \{a,b\}^*a$ has last letter $a$. 

For further use we mention also the trivial fact that
$u\sim v$ implies either that $u$ is a (not necessarily proper)
prefix of $v$ or that $v$ is a prefix of $u$. 

\begin{lemma}\label{lemequivinterval}
Each equivalence class for $\sim$ is an interval of the 
lexicographically ordered poset $\{a,b\}^*$.
\end{lemma}

\begin{proof}
It is enough to show that $u<v<ua^i$ implies $v=ua^h$
with $h\in\{1,\dots,i-1\}$. The easy verification is left to the reader.
\end{proof}

Lemma \ref{lemequivinterval} shows that we can apply the previous construction.
Thus we obtain the {\em twisted alphabetical order} which we denote
by $\prec$. In summary, $u\prec v$ if and only if either $u,v$ are both in a common equivalence class $wa^*$ and $v<u$, or if they belong to
two different equivalence classes and $u<v$.
Equivalently, $u\prec v$ if either $\mu_a(u)\not=\mu_a(v)$ and $u<v$ or if
$\mu_a(u)=\mu_a(v)$ and $u\in vaa^*$ where $\mu_a$ is extended to all elements
of $\{a,b\}^*$ by setting $\mu_a(w)=w$ if $w\not\in\{a,b\}^*a$ (ie. $\mu_a$
erases always a final maximal (perhaps empty) string of 
consecutive letters $a$ in a word).

The following result summarizes a few properties of the twisted order:

\begin{lemma}\label{lemproptwistorder}
(i) The two order relations $<$ and $\prec$ induce opposite orders on an equivalence class of $\sim$.

(ii) If $u,v$ are not in the same equivalence class of $\sim$, then $u<v$ if and only if $u\prec v$ and this depends only on the equivalence classes of $u$ and $v$.

(iii) The restriction of the map $\mu$ defined in Section \ref{sectprefixfreeclosed} to the set $\{a,b\}^*a$ is order-preserving in the following sense:
For two elements $u,v$ in $\{a,b\}^*a$ with $u<v$ we have 
either $\mu(u)=\mu(v)$ (and they are in a common class $\mu(u)a^*$)
or we can apply (ii) above.  
\end{lemma}

We leave the easy proof to the reader. \hfill$\Box$

Call {\em left branch} of a subset $L$ of $\{a,b\}^*$ a non-empty 
intersection of $L$ with an equivalence class of $\sim$. 

\subsection{Prefix-free sets and the twisted alphabetical order}
\label{sectpreffreeandtwist}

\begin{lemma}\label{prop}
Let $C$ be a finite maximal prefix-free set in $\{a,b\}^*$ with associated prefix-closed set $P$. Given an element $c$ in $C$ we denote by 
$c'\in C_a$ the largest lower bound of $c$ in $C_a$ (ie. $c'$
is maximal in $C_a$ such that $c'\leq c$). Then $\mu(d)\preceq 
\mu(c')$ (with $\mu$ defined in Section \ref{sectprefixfreeclosed}) 
for every $d$ in $C$ such that $d\leq c$.
\end{lemma}

An instance of this lemma in our running example is: For $c=bab$ we have $c'=ba^2$. Taking $d=c$ we get $\mu(d)=ba\prec b=\mu(c')$.

\begin{proof}
If $d$ is in $C_a$ this follows from prefix-freeness of $C_a$ and from 
assertion (iii) in Lemma \ref{lemproptwistorder}.

Suppose now that $d\in C_b$. Define $c''\in C_a$ as the unique element of
$C$ in the set $\mu(d)a^*=\mu_b(d)a^*$. If $c''=c'$ then $\mu_a(c')=
\mu_a(\mu(d))\succ \mu(d)$ by definition of $\prec$ on equivalence classes.
If $c''\not=c'$ then $c''<c'$ by maximality of $c'$. The elements
$c''$ and $c'$ define thus two 
different equivalence classes $\mu_a(c'')a^*$ and $\mu_a(c')a^*$ 
and we can apply assertion (ii) of Lemma \ref{lemproptwistorder}.
\end{proof}

\section{Right congruences of a free monoid.}\label{congr}

A {\em right congruence} of a monoid is an equivalence relation $\equiv$ which is compatible with right-multiplication: $u\equiv v$ implies $uw\equiv vw$. Observe that each element in the monoid induces, by right-multiplication, a function from the set of $\equiv$-classes into itself. Equivalently, we get a right action of the monoid on the 
quotient\footnote{A right congruence of a free monoid is essentially the same thing as a deterministic automaton. More precisely, a right congruence of a free monoid
corresponds to an automaton with an initial state but without prescribed set of final states, which is {\em accessible} in the sense that each state can be reached
from the initial state.}.
Recall the well-known bijection between right congruences of finite index in a free  monoid $A^*$ and triplets $(C,P,f)$ where $C$ is a finite maximal prefix-free set with associated prefix-closed set $P$ and where $f:C\rightarrow P$ is a function such that $f(c)\in P$
is alphabetically smaller than $c$ for every $c$ in $C$.
The corresponding congruence is generated by the relations $c\equiv f(c)$ for
$c$ in $C$. The prefix-closed set $P$ is a set of representatives for
the quotient set 
$A^*/\equiv$. The right action of $A^*$ on the quotient is completely defined
as follows: A letter $x$ of the alphabet $A$ acts on $p$ in $P$
by $p.x=px$ 
if  $px$ is in $P$ and by $p.x=f(px)$ otherwise (see for example Proposition 7
of \cite{BR}).

We illustrate this with the right congruence defined by $a^2\equiv 1, ba^2\equiv b,b^2 a\equiv b^2,ab\equiv1,bab\equiv ba,b^3\equiv a$ with $C$ and $P$ as in 
Figure \ref{pfset}. Right-multiplication $w\longmapsto
w.a$ or $w\longmapsto w.b$ by $a$ or $b$ on the set $P=\{1,a,b,ba,b^2\}$ is given by
$$\begin{array}{|c||c|c|c|c|c|}
\hline
w&1&a&b&ba&b^2\\
\hline\hline
w.a&a&1&ba&b&b^2\\
\hline
w.b&b&1&b^2&ba&a\\
\hline\end{array}\ .$$

\section{Left-to-right maxima and indecomposable permutations}\label{sectLRmaxima}

A {\em left-to-right maximum} of a permutation $\sigma\in S_n$ is a value $\sigma(i)=j\in\{1,\ldots,n\}$ such that $\sigma(h)<j$ for $h<i$. We call $i$ the {\em position} and $j$ the {\em value} of the left-to-right
maximum $\sigma(i)$. 
The following result is well-known.

\begin{lemma}\label{left-to-right}
A permutation $\sigma\in S_n$ with successive positions $i_1<\dots <i_k$
of left-to-right maxima is indecomposable (in the sense of Section \ref{main}) if and only if $\sigma(i_j)\geq i_{j+1}$ for $j=1,\ldots,k-1$.
\end{lemma}
Observe that one has always $i_1=1$ and $\sigma(i_k)=n$ with these notations.

For later use, we state and prove the following result, which holds for any permutation expressed as a word $w=a_1\cdots a_n$ involving $n$ distinct letters of a totally
ordered alphabet. We denote by $st(w)=i_1\cdots i_n\in S_n$ 
the associated {\em standard permutation} of $w$ obtained by replacing each letter $a_j$ in $w$ by its rank $i_j$ 
in the totally ordered set $\{a_1,\dots,a_n\}$.
The standard permutation of $w=3649$ for example is given by $st(w)=1324$.

\begin{lemma}\label{ptheta}
Let $\theta\in S_n$ have successive left-to-right maxima in positions $i_1,\ldots,i_k$, with values $j_1,\ldots,j_k$. Let $\sigma=st(w)$, where $w=\sigma(2)\dots\sigma(i_2-1)\sigma(i_2+1)\dots\sigma(i_k-1)\sigma(i_k+1)\dots\sigma(n)$
is obtained from the word $\theta=\sigma(1)\dots\sigma(n)$
by removal of the left-to-right maxima $j_1,\cdots,j_k$. Then  
$$
p(\theta)=p(\sigma)+kn-\frac{k(k+1)}{2}+\sum_{1\leq s\leq k} (j_s-i_s)
$$
where $p(\sigma)$ is the cardinality of the union of all hooks (see Section \ref{hook}).
\end{lemma}

As an example, consider $\theta=\underline 32\underline 54\underline 61$, with
underlined left-to-right maxima. We have 
$i_1=1$, $i_2=3$, $i_3=5$, $j_1=3$, $j_2=5$, $j_3=6$, $w=241$, $\sigma=st(w)=231$. The matrices of $\theta$ and $\sigma$ (with hooks represented by 
boldfaced 0's) are
$$
\left(
\begin{array}{ccccccc}
\bf 0&\bf 0&1&0&0&0\\
\bf 0&1&\bf 0&0&0&0\\
\bf 0&\bf 0&\bf 0&\bf0&1&0\\
\bf 0&\bf 0&\bf 0&1&\bf 0&0\\
\bf 0&\bf 0&\bf 0&\bf0&\bf 0&1\\
1&\bf 0&\bf 0&\bf 0&\bf0&\bf 0
\end{array}
\right)
\hbox{ and }
\left(
\begin{array}{ccccccc}
\bf 0&1&0\\
\bf 0&\bf 0&1\\
1&\bf 0&\bf 0
\end{array}
\right).
$$
We see that $p(\theta)=22$, $p(\sigma)=5$, $k=3$, $n=6$. The equality 
$22=5+3\cdot 6-\frac{3\cdot 4}{2}+(3-1)+(5-3)+(6-5)$ illustrates Lemma \ref{ptheta}.

\begin{proof}
The matrix of $\theta$ looks like in the figure below, where the 1's represent the left-to-right maxima, where the $*$'s represent the possible positions of the other 1's, and where the north-eastern region (which is empty on the figure) has only 0's. Observe that the matrix of $\sigma$ is the submatrix obtained by removing the rows and the columns containing all 1's displayed in the figure and corresponding to all left-to-right maxima. 
The definition of $p(\sigma)$ (given in Section \ref{hook}) shows that the difference $p(\theta)-p(\sigma) $ is equal to the number of 0's in the figure. Since the row (resp. column) coordinates of the left-to-right maxima are $i_1,\ldots,i_k$ (resp. $j_1,\ldots,j_k$), this number of 0's is obtained by summing up the  rows of 0's, columns of 0's, and by subtracting the 0's at intersections 
(which have been counted twice). Thus it is
$j_1-1+j_2-1+\cdots+j_k-1+(n-i_1)+(n-i_2)+\cdots+(n-i_k)-(1+2+\cdots+k-1)$
which gives the formula of Lemma \ref{ptheta}.
$$
\begin{array}{cccccccccccccccccc}
0&\cdots&0&1\\
*&\cdots&*&0\\
\vdots&\vdots&\vdots&\vdots\\
*&\cdots&*&0 \\
0&\cdots&0&0&0&\cdots&0&1\\
*&\cdots&*&0 &*&\cdots&*&0\\
\vdots&\vdots&\vdots&\vdots&\vdots&\vdots&\vdots&\vdots\\
*&\cdots&*&0&*&\cdots&*&0\\
0&\cdots&0&0&0&\cdots&0&0&\cdots&\cdots&0&\cdots&0&1\\
*&\cdots&*&0 &*&\cdots&*&0&\cdots&\cdots&*&\cdots&*&0\\
\vdots&\vdots&\vdots&\vdots&\vdots&\vdots&\vdots&\vdots&\cdots&\cdots&\vdots&\vdots&\vdots&\vdots&\\
*&\cdots&*&0&*&\cdots&*&0&\cdots&\cdots&*&\cdots&*&0\\
\end{array}
$$
\end{proof}

\section{Indecomposable permutations and regular right congruences of the free monoid $\{a,b\}^*$.}

\subsection{Indecomposable permutations, regular right congruences and subgroups of the free group}
A right congruence of a monoid $\mathcal M$ 
is {\em regular} if right multiplication $\mathcal M\ni x\longmapsto
xa$ by an arbitrary element $a\in \mathcal M$ 
induces a bijection on the quotient. (It is of course enough to consider 
right-multiplications by elements in a set of generators.)
Regular right congruences are {\em right-simplifiable} congruences, meaning that $uw\equiv vw$ implies $u\equiv v$. The two properties are equivalent when the index is finite. (Right-multiplication on classes of 
right-simplifiable congruences of infinite index 
induce injections which are in general not bijective as shown by the 
example $\mathcal M=\{a,b\}$ endowed with the right-simplifiable congruence
defined by $u\equiv v$ if $a^*u\cap a^*v\not=\emptyset$. Indeed, 
right-multiplication by $b$ fails to yield the class $a^*$ represented by the
empty word.)

Regular right congruences of index $n$ in a free monoid are in bijection with subgroups of index $n$ of the free group on the same alphabet. Such a subgroup $H$ gives indeed rise to the right congruence $u\equiv v$ if and only if $Hu=Hv$. This yields the desired mapping from the set of subgroups of index $n$ onto the set of regular right congruences of the same index. Conversely, a regular right congruence of index $n$ defines a right action by bijections of the free monoid on the quotient set. This action extends uniquely to a transitive action of the free group, and the stabilizer of the class of the neutral element is a subgroup of index $n$. This gives the bijection.

This bijection has the following classical topological interpretation:  
Subgroups of the free group $\langle a,b\rangle$ generated by $a$ and $b$ (the case of arbitrary free groups is analogous) correspond to isomorphism classes of connected coverings
$(\tilde \Gamma,v_*)$ with a marked base-vertex $v_*$ of the connected graph $\Gamma$
given by of two oriented loops labelled $a$ and $b$ attached 
to a unique common vertex.
The fundamental group of $\Gamma$ is of course the free group $\langle a,b\rangle$ consisting of all reduced words in $\{a^{\pm 1},b^{\pm 1}\}^*$.
The  fundamental group of $(\tilde \Gamma,v_*)$ is the 
subgroup of all elements in $\langle a,b\rangle$ which
lift to closed paths of $\tilde \Gamma$ starting and ending at $v_*$.
If $\tilde \Gamma$ is finite (or more generally if the right actions 
of the cyclic groups $\langle a\rangle$ and $\langle b\rangle$ 
on the right cosets of $\pi_1(\tilde \Gamma,v_*)$ have only finite orbits), 
one can join an arbitrary initial vertex $\alpha$ of $\tilde \Gamma$ to an arbitrary final vertex $\omega$ by a path 
corresponding to a word in the free monoid $\{a,b\}^*$, 
ie. by a path using only positively oriented edges. 
In particular, such a graph $\tilde \Gamma$ has a canonically defined spanning tree
$P=\cup_{v\in V(\tilde \Gamma)}p_v$ where $p_v$ is the lexicographically smallest path
labelled by a word in $\{a,b\}^*$ which joins the marked vertex $v_*$ of $\tilde \Gamma$ to a given vertex $v$ of $\tilde \Gamma$. The set 
$P$ is prefix-closed and the remaining set of labelled oriented edges in 
$\tilde \Gamma$ defines a regular right congruence. The corresponding prefix-free set $C=P\{a,b\}\setminus P$ indexes
a free generating set of $\pi_1(\tilde \Gamma,v_0)$ by associating
the generator $cp_c^{-1}$ to every element $c$ where $p_c$ is the unique
representant $p_c\in P$ defining the same vertex as $c$ in $\tilde \Gamma$.  

Subgroups of finite index of a free group were first counted by Marshall Hall Jr. \cite{H}. 
The values for the number $c_n$ of subgroups of index $n$ in the free group $F_2=\langle a,b\rangle$ on two generators, or equivalently for the number of regular right congruences of index $n$ in the free monoid $\{a,b\}^*$, are $1,3,13,71,461, 3447$ for $n=1,2,3,4,5,6$, see \cite{OEIS}, sequence A3319. Remarkably, $c_n$ is equal to the number of indecomposable permutations in $S_{n+1}$. The symmetric group $S_3$ for example contains $3$ indecomposable permutations given by  $312,213,321$ and the $3$ subgroups of index 2 in $\langle a,b\rangle$ are $\langle aa,ab,ba\rangle, \langle a, bab, bb\rangle, \langle aa,aba,b\rangle$. A first bijection between the set of subgroups of index $n$ of the free group $\langle a,b\rangle$ on two generators and indecomposable permutations in $S_{n+1}$ was given by Dress and Franz \cite{DF1}. Other bijections were discovered later by Sillke \cite{Si}, Ossona de Mendez and Rosenstiehl \cite{OR}, and Cori \cite{Cr}.

\subsection{Another bijection}

We describe a further bijection, useful for proving Theorem \ref{mainr}. Since a regular right congruence $\equiv$ of 
$\{a,b\}^*$ is a particular case of a right congruence, it can be described 
by a triplet $(C,P,f)$ as in Section \ref{congr}. Regularity,
equivalent to bijectivity of the right action on the quotient 
represented by $P$, implies that we have $f(C_a) \subset P_b$ and $f(C_b)\subset P_a$, where, as previously, $C_l=C\cap\{a,b\}^*l$ and $P_l=P\cap(\{a,b\}^*l\cup\{1\})$
for $l$ in $\{a,b\}$.
Indeed, $f(ua)=va$ for $ua\in C_a$ and $va\in P$
implies $u\equiv v$ by right simplification in contradiction with the
fact that elements of $P$ represent non-equivalent classes.
The case of $f(ub)=vb$ is ruled out similarly.

Moreover, 
the inequality $f(c)<c$ for any $c\in C$ implies recursively that 
$f(c)=\mu_a(c)$ for any $c\in C_a$. This shows in particular that 
$f$ induces a bijection from $C_a$ onto $P_b$. 
For instance, looking at the running example (see Figure \ref{pfset}, noticing that the alphabetical order is read there by turning counterclockwise around the tree, starting from the root), we must have $f(a^2)<a^2$ and $f(a^2)\in P_b$, hence $f(a^2)=1$; then $f(ba^2)<ba^2$ and $f(ba^2)\in P_b$, hence $f(ba^2)=b$; similarly, $f(b^2a)=b^2$. 

The restriction of $f$ to $C_b$ determines thus the regular 
right congruence $\equiv$ completely. This restriction
is a bijection from $C_b$ onto $P_a$: Indeed, the two sets have the same cardinality. Moreover $f$ is injective, since if $ub,vb\in C_b$ and $f(ub)=f(vb)$, then $ub\equiv vb$ so that by regularity $u\equiv v$, then $u=v$ since $P$ is a set of representants of the quotient, and finally $ub=vb$.

Since the intersection $P_a\cap P_b$ is reduced to the empty word $1$, the map
$f$ from $C$ to $P$ is almost a bijection: it is surjective and 
each element of $P$ has a unique preimage, except the empty word which has exactly two preimages:  
a unique preimage $a^k=a^*\cap C_a$ in $C_a$ and a unique preimage 
$f^{-1}(1)\cap C_b$ in $C_b$.

We introduce now the set $\tilde P=P\cup \{a^{-1}\}$ and we consider the bijection $\phi$ from $C$ onto $\tilde P$ which coincides with $f$ except that $\phi(a^k)=a^{-1}$ where $a^k=a^*\cap C$. 

The twisted alphabetical order is extended to $\{a,b\}^*\cup\{a^{-1}\}$ by the rule: $a^i\prec a^{-1}\prec w$ for any $i\geq 0$ and for any word $w\in a^*b\{a,b\}^*$ involving $b$.

Writing $C=\{c_1 < c_2 < \ldots < c_{n+1}\}$ and $\tilde P=\{p_1\prec p_2\ldots \prec p_{n+1}\}$ we get a unique permutation $\theta\in S_{n+1}$ such that $\theta(i)=j$ if $\phi(c_i)=p_j$. 

Note that the twisted alphabetical order $\prec$ and the alphabetical order $<$ coincide on the prefix-free set $C$.

\begin{theorem}\label{bijection}
The map $\equiv \, \mapsto \theta$ is a bijection from the set of regular right congruences on $\{a,b\}^*$ into $n$ classes
onto the set of indecomposable permutations in $S_{n+1}$.
\end{theorem}

We illustrate Theorem \ref{bijection} by considering the sets $C,P$ together with the right congruence defined by $a^2\equiv 1, ba^2\equiv b,b^2 a\equiv b^2,ab\equiv1,bab\equiv ba,b^3\equiv a$ in our running example. We have $C=\{a^2<ab<ba^2<bab<b^2a<b^3\}$ and $f(a^2)=1,f(ba^2)=b,f(b^2a)=b^2$ and $f(ab)=1,f(bab)=ba,f(b^3)=a$. Therefore, with $\tilde P=\{a\prec 1 \prec a^{-1}\prec ba\prec b\prec b^2\}$, we have $\phi=f$ except that $\phi(a^2)=a^{-1}$. Therefore $\phi$ is the bijection $C\rightarrow \tilde P$ 
represented by the following array of two rows:
$$
\phi=\left(\begin{array} {cccccc}a^2&ab&ba^2&bab&b^2a&b^3\\a^{-1}&1&b&ba&b^2&a\end{array}\right).
$$
Replacing words of $\phi$ by their respective position for the lexicographical order $a^2<ab<ba^2<bab<b^2a<b^3$ 
on the prefix-free set $C$, respectively for the twisted lexicographical order 
$a<1<a^{-1}<ba<b<b^2$ on $\tilde P$, 
we get the permutation
$$
\theta=\left(\begin{array} {cccccc}1&2&3&4&5&6\\3&2&5&4&6&1\end{array}\right).
$$
\begin{proof} 
We prove first that the permutation $\theta$ (associated to a
regular right congruence with $n$ classes in $\{a,b\}^*$) is 
an indecomposable element of $S_{n+1}$.
The equivalence class $\mu_a(f(c))a^*$ of an element
$c$ in $C_b$ intersects $C_a$
in a unique element $c'$ which is lexicographically smaller
than $c$. Since 
$$\phi(c')=\left\lbrace\begin{array}{ll}
\mu_a(c')\qquad &\hbox{if }c'\not\in a^*\\
a^{-1}&\hbox{if }c'=C_a\cap a^*
\end{array}\right.$$ 
is the maximal element (with respect to the twisted order, extended to $\tilde P$) in the equivalence
class of $f(c)$,  
left-to-right maxima of $\theta$ correspond to 
a subset of $C_a$. The equality $f(c)=\mu_a(c)$
implies that all elements of $C_a$ define left-to-right maxima.
We apply now Lemma \ref{left-to-right} for proving 
indecomposability of $\theta$ as follows:
Given an element $c\in C_a$, the element $\phi(c)$
is always the largest element with respect to the twisted order
of the set 
$$\tilde P(< c)=\{p\in \tilde P,p< c\hbox{ lexicographically}\}\subset
P\cup\{a^{-1}\}$$
where $a^{-1}$ is by convention the lexicographically smallest element
of $\tilde P$. Indecomposability of $\phi$ amounts thus by Lemma
\ref{left-to-right} to the inequality
\begin{eqnarray}\label{ineqCP}
\vert C(\leq c)\vert&<&\vert \tilde P(<c)\vert
\end{eqnarray}
for all $c\in C_a$ where
\begin{eqnarray*}
\tilde P(< c)&=&\{p\in \tilde P,p< c\hbox{ lexicographically}\},\\
C(\leq c)&=&\{c'\in C_a,c'\leq c\hbox{ lexicographically}\}.
\end{eqnarray*}
The identity
$$\vert \tilde P(<c)\vert=\vert P(<c)\vert+1$$
where 
$$P(< c)=\{p\in P,p< c\hbox{ lexicographically}\}=\tilde P(<c)\setminus\{a^{-1}\}$$
shows that the strict inequality (\ref{ineqCP}) amounts to 
$$\vert C(\leq c)\vert\leq \vert P(<c)\vert$$
for all $c\in C_a$. This inequality
holds since $\mu$ restricts to an injection 
from $C(\leq c)$ into $P(<c)$ for all $c$ in $C\setminus(C\cap b^*)$,
as observed in Section \ref{sectprefixfreeclosed}.

Thus we have a map associating an indecomposable permutation $\theta$ 
to every regular right congruence $\equiv$. It is now enough to establish
injectivity of this map. Surjectivity follows then from the known equicardinality of the two involved sets. 
The cardinality of $C_a$ equals the number of left-to-right maxima of $\theta$. Let $\theta$ have successive left-to-right maxima in positions $i_1,\ldots,i_k$, with values $j_1,\ldots,j_k$. The $i_h$ are the ranks in the totally ordered set $C$ of the elements of $C_a$. The lengths $l_i$ of the left branches of $P$ are determined by the differences between the values of two successive such maxima: $l_1=j_1-1$ and $l_i=j_i-j_{i-1}$ if $i\geq 2$. 
By Lemma \ref{charactree}, the tree defined by the maximal prefix-free set $C$ 
and its associated prefix-closed set $P$ are thus
completely determined by the numbers $i_1,\dots,i_k$ and $j_1,\dots,j_k$. 
From $C$ and $P$ we immediately recover the function $f$ on $C_a$. The bijection
$f:C_b\longrightarrow P_a$ is encoded by the standard permutation $st(\theta)$
(as defined in Lemma \ref{ptheta} of Section \ref{sectLRmaxima})
of $\theta$. The equivalence relation $\equiv$ is thus 
completely determined by $\theta$.
\end{proof}

{\bf Remark.} It is not difficult to invert the map $\equiv\  \mapsto \theta$
of Theorem \ref{bijection}. Indeed, positions 
and values of left-to-right maxima of an indecomposable permutation 
$\theta\in S_{n+1}$ determine a unique maximal prefix-free set $C$ having
$n+1$ elements. The associated standard permutation $st(\theta)$
encodes a regular right congruence given by a suitable map from $C$
into the set $P$ of all proper prefixes of $C$. This avoids 
equicardinality results and
gives a bijective proof of Theorem \ref{bijection}.

\subsection{Fixing $C$ and $P$.}\label{FixCP}
We fix a finite maximal prefix-free set $C$ of $n+1$ elements in $\{a,b\}^*$ 
with $C_a=C\cap\{a,b\}^*a$ containing $k$ elements.
Consider the set of all regular right congruences into 
$n$ equivalence classes with 
lexicographically smallest representants given by the prefix-closed set $P$
associated to $C$. Theorem \ref{bijection}
gives by restriction a bijection between the set of these congruences and the set of bijections $\alpha:C_b\rightarrow P_a$ satisfying $\alpha(c)<c$: We have indeed $\alpha(c)=\phi(c)=f(c)<c$ for any $c\in C_b$. Since $C_b=C\setminus C_a$ 
has $n+1-k$ elements and is 
totally ordered by $<$ and since $P_a$ is totally ordered by $\prec$, the 
bijection $\alpha$ is naturally associated to a permutation $\sigma$
of $S_{n+1-k}$. This permutation is simply the standard permutation of the indecomposable permutation $\theta$ (encoding a regular right representation with $n$ classes). 
It is obtained from $\theta$, viewed as a word, by removing the values of all left-to-right maxima, see Section \ref{sectLRmaxima}. As already observed earlier, 
positions and values of left-to-right maxima of the permutation $\theta$
encode the underlying finite maximal prefix-free set $C$. 

We use assertion (ii) of Lemma \ref{lemproptwistorder}
for ordering (alphabetically) left branches of $P$. We denote 
by $l_1,\ldots,l_k$ the corresponding lengths and we set $s_i=l_1+\cdots +l_i$. 

We observe for later use the following facts: If $i_1,\ldots,i_k$, $j_1,\ldots, j_k$ are as in Lemma \ref{ptheta} applied to a permutation in $S_{n+1}$, 
then the previous proof implies
$C_a=\{c_{i_1}<\ldots<c_{i_k}\}$ and $\{p_{i_1}\prec \ldots \prec p_{i_k}\}=\{\phi(c_{i_1})\prec \ldots \prec \phi(c_{i_k})\}$. Hence $j_h$ is the rank of $\phi(c_{i_h})$ in the set $\tilde P$ ordered by $\prec$. Observe that $\phi(c_1)=a^{-1}$ and $\phi(c_{i_h})=f(c_{i_h})=\mu_a(c_{i_h})$ if $h\geq 2$. Thus $j_1=l_1+1=s_1+1$ and, more generally, $j_h=s_h+1$ for any $h=1,\ldots,k$.
Lemma \ref{ptheta}, with $n$ replaced by $n+1$, shows
$$p(\theta)=p(\sigma)+(n+1)k-\frac{k(k+1)}{2}-\sum_hi_h+\sum_hs_h+k$$
which simplifies to
\begin{eqnarray}\label{formulaptheta}
&p(\theta)=p(\sigma)+(n+1)k-\frac{k(k-1)}{2}-\sum_hi_h+\sum_hs_h.
\end{eqnarray}

\section{Right ideals in $\mathbb F_q\langle a,b\rangle$.}\label{CPalpha}

It follows from \cite{R} (see also \cite{BR}, Proposition 7.1) that the set of right ideals of codimension $n$ of the free non-commutative associative algebra 
$\mathbb F_q\langle a,b\rangle$ over $\mathbb F_q$ generated by $a$ and $b$
is in bijection with the set of triplets $(C,P,(\alpha_{c,p}))$, where $C$ is a finite maximal prefix-free set with associated prefix-closed set $P$, with $P$ of cardinality $n$ and $C$ of cardinality $n+1$, and where $(\alpha_{c,p})$ is a family of elements in $\mathbb F_q$ with $c\in C$, $p\in P$ and $p<c$ for the alphabetical order.

In this case the right ideal $I$ is generated by the polynomials $c-\sum_{p<c}\alpha_{c,p}p$. These polynomials are in fact free generators of the right $\mathbb F_q\langle a,b\rangle$-module $I$. Moreover, the elements of $P$ are representatives of an $\mathbb F_q$-basis
of the quotient $I\backslash\mathbb F_q\langle a,b\rangle$ and the right action of $\mathbb F_q\langle a,b\rangle$ on the quotient is completely defined by 
$$p.x=\left\lbrace\begin{array}{cl}px\qquad&\hbox{if }px\in P\\
\sum_{q<c}\alpha_{c,q}q\qquad&\hbox{if }c=px\in C\end{array}\right.$$
 where $p$ is in $P$ and $x\in\{a,b\}$ is a letter of the alphabet.

The matrices $\mu(a),\mu(b)$ of the right action of $a$ and $b$ 
with respect to the basis $P$ of $I\backslash\mathbb F_q\langle a,b\rangle$
are therefore $P\times P$ matrices with coefficients $\mu(x)_{p,q}$
defined by
\begin{eqnarray}\label{formulamuxpq}
\mu(x)_{p,q}&=&\left\lbrace\begin{array}{cl}
1\qquad&\hbox{if }px=q\\
0\qquad&\hbox{if }px\in P,\ px\neq q\\
\alpha_{c,q}\qquad&\hbox{if }px=c\in C,\ q<c\\
0\qquad&\hbox{if }px=c\in C,\ q>c.\end{array}\right.
\end{eqnarray}

\section{Right ideals in $\mathbb F_q \langle 
a,b,a^{-1},b^{-1}\rangle$.}

A right ideal $I$ of codimension $n$ in  $\mathbb F_q\langle a,b,a^{-1},b^{-1}\rangle$ determines a right ideal $I\cap \mathbb F_q\langle a,b\rangle$ of codimension $n$ in $\mathbb F_q\langle a,b\rangle$.
Right ideals in $\mathbb F_q\langle a,b\rangle$ determined in this way by right ideals
in $\mathbb F_q\langle a,b,a^{-1},b^{-1}\rangle$ are right ideals $J$ of $\mathbb F_q\langle a,b\rangle$ such that the actions of $a$ and $b$ on the quotient
$J\backslash \mathbb F_q\langle a,b\rangle$ 
are both linear isomorphisms. We obtain in this way a bijection between the set of right ideals of codimension $n$ in $\mathbb F_q\langle a,b,a^{-1},b^{-1}\rangle$ and the set of right ideals of codimension $n$ in $\mathbb F_q\langle a,b\rangle$ such that $a$ and $b$ act both bijectively on the quotient.

\section{q-Count with fixed prefix-free set $C$}\label{sectq-count}

We consider a fixed maximal prefix-free set $C$ of cardinality $n+1$ 
in $\{a,b\}^*$ with associated prefix-closed set $P$ of cardinality $n$. We count all right ideals of codimension $n$ in $\mathbb F_q\langle a,b\rangle$
such that right-multiplication by $a$ and right-multiplication by $b$ induce bijections of the quotient. Such right ideals correspond to triplets $(C,P,(\alpha_{c,p}))$ with $\alpha_{c,p}$ encoding two invertible
matrices $\mu(a)$ and $\mu(b)$ by Formula (\ref{formulamuxpq}) of Section \ref{CPalpha}.

\subsection{Counting the matrices $\mu(a)$.}\label{mua}
The definition of $\mu(a)$ shows that this matrix has a lower triangular block decomposition, with blocks ordered as in Subsection \ref{FixCP} and with block-sizes equal to the lengths $l_i$ of the left branches of $P$. Moreover, each diagonal block is a companion matrix of size $l_i\times l_i$, $i=1,\ldots, k$.
 Strictly lower triangular blocks are filled with 0's except for their 
last row which is arbitrary. In other words, only rows of index $s_1,s_2,\dots$ have some freedom: The first $s_i$ entries of row $s_i$ are arbitrary, except that one of them (in column $s_{i-1}+1$) must be nonzero. Thus there are $(q-1)q^{l_i-1}$ possible choices for the $i$-th diagonal block. This amounts to
$(q-1)^kq^{N}$ possibilities for the matrix $\mu(a)$, where 
\begin{eqnarray}\label{formulaN}
N&=&\sum_{i=1,\ldots,k}(s_i-1)=\sum_{i=1,\ldots,k}s_i-k
\end{eqnarray}
and where $s_i=l_1+\dots+l_i$ is the rank in $P$
of the unique element $p_{s_i}$ in $P$ such that $p_{s_i}a$ is the $i-$th
smallest element of $C_a$.

In our running example given by Figure \ref{pfset}, the matrix $\mu(a)$ is of 
the form
$$
\begin{array}{ccccccccc}
&&1&a&b&ba&b^2\\
&&&&&&\\
1&& 0&1 &0&0&0 \\
a&& *&\times &0&0&0\\
b&& 0&0&0&1&0 \\
ba&&  \times&\times& *&\times &0\\
b^2&& \times&\times&\times&\times&*   
\end{array}
$$
where $*$ represents a nonzero element of the field, whereas $\times$ is any element. This matrix is block-triangular, with 3 diagonal blocks, of size $l_1=2, l_2=2, l_3=1$. There are $(q-1)^3q^8$ such matrices, as predicted by the formula
for $k=3$ and $s_1=l_1=2,s_2=l_1+l_2=4,s_3=l_1+l_2+l_3=5$ leading to $N=(s_1-1)+(s_2-1)+(s_3-1)=1+3+4=8$.

\subsection{Counting the matrices $\mu(b)$.}
Define the auxiliary order $<_b$ by $p<_b q$ if $pb<qb$. We order the rows of $\mu(b)$ with $<_b$ and its columns by $\prec$. 

We consider the partition $\lambda$ with parts 
\begin{eqnarray}\label{formulalambdac}
\lambda_c&=&\vert \{q\in P_a,q<c\}\vert
\end{eqnarray}
indexed by elements
$c$ in $C_b$. A part $\lambda_c$ indexed by $c\in C_b$  is thus defined as the
number of lower bounds of $c$ in $P_a$. Then for any $p,p'\in P$ such that  $c=pb,c'=p'b\in C$ and $p<_b p'$, one has $\lambda_c\leq \lambda_{c'}$, since $c<c'$. Observe that $\lambda$ has $\vert C_b\vert$ parts
and a largest part (indexed by the unique 
element $b^h$ of $C\cap b^*$) of length $\vert C_b\vert$, since $C_b$ is in bijection with $P_a$ and since each element in $P_a$ is $<b^h$. In our running example underlying Figure \ref{pfset}, we have $\lambda=2,3,3$, with
$\lambda_{ab}=2$ (corresponding to $1,a<ab$) and $\lambda_{bab}=\lambda_{b^3}=3$ (since $1,a,ba<bab,b^3$).

A row of $\mu(b)$ indexed by $p\in P$ such that $pb\in P$ has all coefficients zero except for a unique coefficient $1$ with column-index $pb$.
Possibly nonzero entries in column $pb$, other than the 1 in row $p$, are in the rows $q$ with $qb\in C$ and $pb<qb$, by definition of $\mu(b)$ in Section \ref{CPalpha}; they are located below row $p$. Their number is therefore equal to the number of $c\in C_b$ such that $pb<c$.

Removing from the matrix $\mu(b)$ all rows indexed by $p\in P$ such that $pb\in P$ and all columns indexed by $pb\in P_b\setminus\{1\}$, we get a matrix with rows indexed by $C_b$ (since $pb\notin P$ implies $pb\in C_b$) and columns indexed by $P_a$. This matrix is a square matrix of size $|C_b|\times|C_b|$ with nonzero entries contained in the set $E_\lambda$ defined at the beginning of Section \ref{thHag} for the partition $\lambda$ with parts $\lambda_c$ defined by (\ref{formulalambdac}). In other words,  $E_\lambda$ is the set of entries $(p,p')$, with $pb\in C_b$ and $p'\in P_a$, such that $p'<pb$.

Observe now that the matrix $\mu(b)$ is invertible if and only if the submatrix above is invertible, since all removed rows have exactly one non-zero
entry in distinct columns.

Thus, we obtain a total number $q^MH_\lambda(q)$ of possible matrices $\mu(b)$, where $M=|\{(p,c)\in (P_b\setminus \{1\})\times C_b,p<c\}|$ and where $H_\lambda$ is as in Section \ref{hag}. 

The matrix $\mu(b)$ of our running example looks like
$$
\begin{array}{ccccccccc}
&& a& 1& ba& b& b^2\\
&&&&&&\\
 a&& \times& \times & 0&0&0 \\
 1&&  0&0&0&1&0\\
 ba&& \times & \times& \times &\times&0\\
 b&&  0&0&0&0&1 \\
 b^2&&  \times &\times& \times&\times&\times
\end{array}
$$
The associated submatrix obtained by removing both rows and columns containing 1's is the leftmost matrix in the Figure of Section \ref{thHag}.

\section{Proof of Theorem \ref{mainr}}

The two preceeding sections show that the number of right ideals of index 
$n$ in $\mathbb F_q\langle a,a^{-1},b,b^{-1}\rangle$ is given by
$$A_n=\sum_C(q-1)^{k(C)}q^{N(C)}q^{M(C)}H_{\lambda(C)}(q)$$
where the sum is over all maximal prefix-free sets with $n+1$ elements 
in $\{a,b\}^*$ and where $H_\lambda(q)$ counts the number of invertible matrices
with support prescribed by a partition $\lambda$. The numbers $k(C),N(C),M(C)$
and the partition $\lambda(C)$ associated to a maximal prefix-free set $C$ are defined as in Section \ref{sectq-count}.

Applying Haglund's Theorem (Theorem \ref{hag}) we get
\begin{eqnarray*}
&A_n=(q-1)^{n+1}\sum_Cq^{N(C)+M(C)}\sum_{\sigma\in S(\lambda(C))}q^{p(\sigma)}
\end{eqnarray*}
where $S(\lambda(C))$ denotes the set of all permutations with permutation-matrices supported by the partition $\lambda(C)$.

Fixing $C$ we observe that the bijection $\alpha$ of subsection \ref{FixCP}, viewed as a matrix indexed by $C_b\times P_a$, has nonzero entries only in $E_\lambda$ (after identification of $C_b$ with the set of $p\in P$ such that $pb\in C$).
Moreover, we have seen that $\alpha$ may be identified with $\sigma$.
Using the definition $N=-k+\sum_{j=1}^k s_j$ (cf. Formula (\ref{formulaN}))
and the equality
$$M=-\sum_{j=1}^ki_j+(n+1)(k-1)+k-\frac{k(k-1)}{2}$$
given by Lemma \ref{numbers}, we have
\begin{eqnarray*}
A_n&=&(q-1)^{n+1}\sum_C\sum_\sigma q^{\sum_{j=1}^k(s_j-i_j)+(n+1)(k-1)-k(k-1)/2+p(\sigma)}\\
&=&(q-1)^{n+1}\sum_C\sum_\theta q^{p(\theta)-(n+1)}
\end{eqnarray*}
where the last identity is given by Formula (\ref{formulaptheta}) of Subsection \ref{FixCP} and where the second sum is over all possible permutations $\theta$
as in Subsection \ref{FixCP}. Since $\theta$ is necessarily indecomposable
and since an indecomposable permutation $\theta$ of $S_{n+1}$
determines $C$ uniquely,
the first sum can be dropped. This shows the equality
$$A_n=(q-1)^{n+1}\sum_{\theta\in \mathrm{Indec}_{n+1}} q^{p(\theta)-(n+1)}.$$
A comparision with 
formula (\ref{formulemainreq}) ends the proof.\hfill$\Box$

\section{Conclusion}
Our main result can also be interpreted as a cellular decomposition of the set of right ideals of codimension $n$ in $\mathbb F\langle a,a,a^{-1},b^{-1}\rangle$
over an arbitrary field $\mathbb F$. Cells are indexed by indecomposable permutations of $S_{n+1}$ and the cell corresponding to an indecomposable permutation $\theta$ in $S_{n+1}$ is isomorphic to $(\mathbb F^*)^{n+1}\times \mathbb F^{\frac{(n+1)(n-2)}{2}+inv(\theta)}$.

There exists perhaps an extension of our main result to the ring of Laurent polynomials in $g\geq 3$ variables. Indeed, one ingredient of our proof is a bijection between subgroups of index $n$ of the free group in 2 generators and indecomposable permutations in $S_{n+1}$ and Dress and Franz have generalized their bijection in \cite{DF1} to a bijection between subgroups of index $n$ of the free group in $g$ generators and systems of $g-1$ indecomposable permutations in $S_{n+1}$, see \cite{DF2}.

\end{document}